\documentclass[11pt,letterpaper]{amsart}
\usepackage{amssymb,
                       amsmath,
                               txfonts,
                                      mathrsfs,
                                              titletoc,
                                                    appendix,
                                                    relsize}
\usepackage{graphicx}
\usepackage{xcolor}
\usepackage{subcaption}

\usepackage[alphabetic]{amsrefs}

 \newtheorem{thm}{Theorem}[section]
 \newtheorem{cor}[thm]{Corollary}
   
  \newtheorem{cjt*}{Conjecture}
 \newtheorem{lem}[thm]{Lemma}
 \newtheorem{prop}[thm]{Proposition}
 \newtheorem{defn}[thm]{Definition}
 \newtheorem{rem}[thm]{Remark}
 \numberwithin{equation}{section}
\newtheorem{lem*}{Lemma}
\newtheorem{cor*}{Corollary}

\newenvironment{pfm2}{\medskip \noindent
{\it Proof of Corollary  \ref{cor}.}}{\hfill $\square$\par
}

\newenvironment{pfm3}{\medskip \noindent
{\it Proof of Theorem  \ref{addthm}.}}{\hfill $\square$\par
}

\newenvironment{pfm}{\medskip \noindent
{\it Proof of Theorem \ref{main}.}}{\hfill $\square$\par
}

\newenvironment{pf}{\medskip \noindent
{\bf Proof of Theorem \ref{mainthm1}.}}{\hfill \rule{.5em}{1em}
\\}

\newenvironment{rmk}{\mbox{ }\\{\bf  Remark}\mbox{ }}{
\hfill $\Box$\mbox{}\bigskip}

\def\la{\lambda}

\def\th{\theta}

\def\R{\Bbb{R}}

\begin{document}

\title{Some configuration results for area-minimizing cones}
\author{Yongsheng Zhang}

\email{yongsheng.chang@gmail.com}
\date{2026/02/16}
\address{Academy for Multidisciplinary Studies, Capital Normal University, Beijing 100048, P. R. China}
\begin{abstract}
We discover some very general configuration results for constructing area-minimizing cones.
In particular, given any closed minimal submanifold in some Euclidean sphere,
every cone over the minimal product of sufficiently many copies of the submanifold turns out to be area-minimizing;
meanwhile every cone over the minimal product of the submanifold and a round sphere of sufficiently large dimension is also area-minimizing.  
Here no additional geometric assumption (e.g. on isometry group or second fundamental form) is required.
Moreover, we establish that the category of regular area-minimizing cones in Euclidean spaces and that of closed minimal submanifolds in
Euclidean spheres share the same cardinality.
\end{abstract}
\maketitle
\section{Introduction}\label{S1} 
    Let $L$ be a set in the unit sphere $\mathbb S^N$ in Euclidean space $(\mathbb R^{{N+1}},g_E)$
    and
    $C(L)$ be the cone $\{tx: t\geq 0,\ x\in L\}$ over its  link $L$.
      It is called area-minimizing  in geometric measure theory, 
                     if the truncated cone $C\bigcap \bold{B}^{{N+1}}(1)$ has least mass
               among all integral 
               currents (see \cite{FF}) with boundary $L$. 
               A fundamental result (Theorem 5.4.3 in \cite{F})
        asserts the existence of tangent cone at every point of an area-minimizing rectifiable current
        and moreover
each tangent cone
        itself area-minimizing.
        Although ambient Riemannian manifold can be quite arbitrary, 
        after blowing-up those cones always live in standard Euclidean spaces.
   Lots of efforts have been made in studying important properties and finding examples  of them in various situations, 
                e.g.,
                \cite{JS, BdGG, BL,  HL, HS, FK, Ch, Law0, Law, NS,  XYZ2, z, TZ} and etc.

               When $L$ is an embedded closed smooth submanifold, 
                     the cone $C(L)$ is smooth away from the origin
                     and we say it is a  regular cone.
In this paper, we search for regular area-minimizing cones and end up with very general configuration results.

Let us get started with a basic question regarding the structure of the category of regular area-minimzing cones.
\begin{quote}
$(\star)$
{\em Given two regular area-minimizing cones $C(L_1)$ and $C(L_2)$ in $\mathbb R^{{N_1+1}}$ and $\mathbb R^{{N_2+1}}$ respectively, 
is the cone over minimal product $L_1\times L_2$  still area-minimizing in $\mathbb R^{{N_1+N_2+2}}$?}
\end{quote}
Here the minimal product of $L_1$ and $L_2$ of dimension $k_1$ and $k_2$ respectively is 
                     \begin{equation}\label{mp}
          \left(\sqrt{\frac{k_1}{{k_1+k_2}}} L_1, \sqrt{\frac{k_2}{{k_1+k_2}}} L_2\right)\subset \mathbb S^{N_1+N_2+1}.
              \end{equation}
Apparently, the minimal product can apply to multiple inputs inductively. Along this line, Lawlor (Theorem 5.1.1, \cite{Law}) completely settled down $(\star)$ 
when each input $L_i$ is $\mathbb S^{N_i}$, with $i=1,2,\cdots, n$ and $N_i\in \mathbb Z_+$.
                     In particular, 
                                    whenever $\Sigma_{i=1}^n N_i>7$,
                                    the corresponding cone is area-minimizing.
                                    In recent work \cite{TZ},
                                     cones of dimension no less than 37 over minimal products of the category of typical minimal submanifolds associated to isoparametric foliations of spheres
                                     are confirmed area-minimizing.
            From these, 
                    one can observe that  the dimension plays an important role for $(\star)$ 
                            and 
                                      that some minimal submanifolds in spheres themselves cannot span area-minimizing cones 
                                      but their minimal products do.
                                      
                                In this paper, among others we get the following.
                                \begin{thm}\label{main}
                                Given embedded closed minimal submanifolds $\{L_i\}$ respectively in Euclidean spheres with $i=1,2,\cdots, n$,
                                every cone over the minimal product $L$ of sufficiently many copies among these $\{L_i\}$ is area-minimizing.
                                \end{thm}

                                    
\begin{rem}
Every closed minimal submanifold with no other assumptions at all can be used as a brick to construct area-minimizing cones in different combinatorial ways.
              The only payment is that the resulting area-minimizing cone may have high dimension.
              \end{rem}
              
              Another neat result is the following.
              
                        \begin{thm}\label{addthm}
                               Let $L$ be a closed smooth minimal subamanifold in some  $\mathbb S^N$.
                              Then the cone over minimal product $L\times \mathbb S^{d}$ is area-minimizing for sufficiently large $d$.
                               \end{thm}
                               \begin{rem}\label{rem14}
                              In fact there exists integer $d_0$  such that the cone over the minimal product $L^k\times \mathbb S^{d_1}\times \cdots \mathbb S^{d_\ell}$  is area-minimizing whenever $d=\sum_{i=1}^\ell d_i\geq d_0$.
                               If $L$ is a minimal  isoparametric hypersurface or a focal submanifold, then this falls into \cite{TZ} and $k+d\geq 36$ would be sufficient for the area-minimality.
                               \end{rem}

                  Prior to this paper, 
          all known regular area-minimizing cones
          are tightly related to either group actions
           or the theory of isoparametric foliations on spheres. 
              The above theorems in particular break these  restrictions.       
              Moreover, 
              a corollary of the theorems from the point of set theory
              is established as follow.
              \begin{cor}\label{cor}
              Regular area-minimizing cones in Euclidean spaces and closed minimal submanifolds in
Euclidean spheres share the same cardinality.
              \end{cor}

              
          The main tool is  Lawlor's curvature criterion for a regular cone to be area-minimizing.
      To be self-contained we briefly explain  his method in \S 2.
          Some useful information about minimal product will be given in \S 3.
          Finally, 
          the proofs of the theorems and corollary and more subtle results  partially answering $(\star)$ and complementing the theorems are collected in \S 4.

{\ }\\
{\bf Acknowledgments. }The author wishes to thank Professor Blaine Lawson for drawing his attention to the powerful work \cite{Law} of Professor Gary Lawlor.
This work was sponsored in part by NSFC (Grant Nos. 12022109, 11971352).
                        
          
          \section{Lawlor's curvature criterion}
          
          Area-minimizing hypercone is a key ingredient to solve the famous Bernstein problem \cite{B} about entire minimal graph of codimension one,
          see  \cite{fle, de2, A}.
           It was Jim Simons who found first non-trivial stable minimal hypercones $C(S^k\times S^k)$ with $k\geq 3$ in \cite{JS} and suspected that they are area-minimizing.
           In the following year, 
                              breakthrough was made in \cite{BdGG} which confirmed those Simons cones' area-minimality and based on this discovery counterexamples
                              of non-planar  minimal graphs to the Bernstein problem were found over domain $\R^N$ when $N\geq 8$.

                            In  \cite{HS}
                          a beautiful characterization of area-minimizing hypercones was established.
             In each side of an area-minimizing hypercone,
             there is a unique dilation-invariant foliation by minimal hypersurfaces.
{Therefore,} the unit norm vector fields lead to a foliation by their integral curves on $\mathbb R^{N+1}\sim 0$,
             which is also invariant under dilation.
             It naturally 
              induces 
              a dual co-degree one calibration form 
                   (see \cite{HL})
              and an area-nonincreasing projection (for surfaces of codimension one) to the hypercone in $\mathbb R^{N+1}\sim 0$.
                            
                        For a given regular minimal cone,    how to determine whether it is area-minimizing or not?
                        The most effective way is to employ the theory of calibration.
                        However, in general, it is not easy to directly find a desired calibration with mild singularity.
                      Lawlor successfully developed a sufficient criterion (sometimes also necessary) 
                                 looking for area-nonincreasing projection structure in suitable angular neighborhood of the minimal cone instead of in the whole space $\mathbb R^N$.
             If the boundary of the neighborhood is mapped to the origin under the projection,
             then one can send everything outside the neighborhood to the origin.
             In this way an area-nonincreasing projection can be globally defined.

              To understand  Lawlor's criterion, let us recall the following with $k<N$.
              
              \begin{defn}\label{Dnr}
              Let $L^k\subset \mathbb S^N$ be a  closed embedded submanifold, $x\in L$
              and
               $S_x$ denote the set of all unit normals of $L$ at $x$.
              Then the normal radius of $L$ 
              at $x$ is 
              $$
             R_x\triangleq \sup\left\{ \, t\in \R_+\,:\, \exp_x(s v_x)\not\in L \, \text{ for all } v_x\in S_x \text{ and } s\in(0,t) \right\}
              $$
              where 
              $\exp_x(s  v_x)=(\cos s) \, x+(\sin s)\, v_x$ forms a normal geodesic along $v_x$.
              The quantity $R(L)\triangleq \inf \{\, R_x\, :\, x\in L\, \}$ is called normal radius of $L$. 
              %
              %
              \end{defn}
              
           For gaining area-minimizing projection, 
                   Lawlor considered
                   in each normal wedge 
                    $$W_x\left(\frac{R_x}{2}\right)=C\left(\exp_x\left(\left[0,\frac{R_x}{2}\right]\cdot  S_x\right)\right)$$
                   foliation  generated by rotations and  dilations of a suitable curve $\gamma_x$.
             For simplicity, all $\gamma_x$ are taken to be the same up to isometry.
             Namely, for very $x\in L$ and $v_x\in S_x$,
             curve $\gamma_x$ in span$\{x,\, v_x\}$
             has uniform expression $r=r(\theta)$ in the polar coordinate.
             Then assembling the projection along the curves provides an area-decreasing map to the $(k+1)$-dimensional $C(L)$
             if and only if
             \footnote{Our dimension of $C(L)$ is $k+1$ instead of $k$ in \cite{Law}.}
            \begin{equation}\label{ineq}
             \frac{dr}{d\theta}\leq r\sqrt{r^{2k+2}\cos^{2k}\theta 
                                                 \inf_{x\in L, \, v\in S_x}
                                                             \left(
                                                             \det
                                                             \left(
                                                                        \textbf{I}-(\tan\theta) \textbf{h}_{ij}^v
                                                                        \right)
                                                             \right)^2
                                                             -1
                                                             }
           \end{equation}
             with 
             $r(0)=\|x\|=1$.
 Here 
             $\textbf{h}_{ij}^v$ means the second fundamental form at $x$ with respect to $v$.
                                                To get narrowest curve $\gamma_x$, 
                                                one needs to require that equality holds everywhere in \eqref{ineq}
                                                and wants to have the smallest $\theta_0$
                                                such that 
                                                $\lim_{\theta\uparrow \theta_0}r(\theta)=\infty$.
                                                As a result, the boundary of
                                                the normal wedge
                                                $W_x\left(\theta_0\right)$ is projected to the origin
                                                and global area-decreasing projection gets set up.
                                                Such $\theta_0$, if existed, is called vanishing angle of $C(L)$.
                                                Note that
                                                there is no focal point of $L$
                                                in the neighborhood 
                                                $\bigcup_{x\in L} W_x\left(\theta_0\right)$
                                                since otherwise that focal point leads to a strictly negative expression under the square root symbol in \eqref{ineq}.
                                                
                                                {\ }
                                                
                                                Lawlor's curvature criterion is the following.
                                                \begin{thm}[Theorem 1.2.1, \cite{Law}]\label{LC}
                                                Given a regular minimal cone $C(L)\subset \R^{N+1}$,
                                                if the vanishing angle $\theta_0$ exists
                                                and $\theta_0\leq \frac{R(L)}{2}$,
                                                then $C(L)$ is area-minimizing.
                                                \end{thm}
                                              
                                                Where does curvature play a role? 
                                                How to check whether a given cone satisfies 
                                                 $\theta_0\leq \frac{R(L)}{2}$?
                                                In practice, 
                                                Lawlor
                                           employed
                                            \begin{equation}\label{control}                                                                                                                                                                              (1-\alpha t)e^{\alpha t}
                                                                       <
                                                                       F(\alpha, t, k+1)
                                                                       \leq
                                                                        \det
\left(
                                                                        \textbf{I}-t \, \textbf{h}_{ij}^v
                                                                        \right)
                                              \end{equation}
                                                      where 
                                                      $t=\tan \theta$, 
                                                                                               $
                                                      \alpha=
                                                      \sup_{x\in L, v\in S_x}
                                                      \|     \textbf{h}_{ij}^v      \|
                                                      $
is the curvature term
                                                       and
                                     \begin{equation}\label{F}   
                                                      F(\alpha, t, k+1)=
                                                                 \left(
                                                                 1-\alpha t\sqrt{\frac{k}{k+1}}
                                                                  \right)
                                                                  \left(
                                                                  1+\frac{\alpha t}{\sqrt{k(k+1)}}
                                                                  \right)^{k},
                                        \end{equation}
                                        to replace the infimum term  in \eqref{ineq}
                                                deriving two kinds of vanishing angles 
                                                      $$
                                                      \theta_c( k+1, \alpha)
                                                      >
                                                      \theta_F( k+1, \alpha)
                                                      \geq
                                                       \theta_0( C(L))
                                                       .
                                                      $$
                                                      So, if either 
                                                      $\theta_F$
                                                      or 
                                                      $\theta_c$ exists of finite value,
                                                      then consequently   vanishing angle  $\theta_0$ must exist.
                                                                                                            For details readers are referred to \cite{Law}.
                                                    
                                                    {\ }
                                                      
                                                     An important   property  about $\th_c$ is the following,
                                                     which provides an effective control on $\th_c$.
                                                      \begin{lem}[Proposition 1.4.2, \cite{Law}]\label{tc}
                                                      For $m,r\in \mathbb Z_+$  and $m>r$, it follows that
                                 \begin{equation}
                                 \label{ratiotc}
                                                   \tan \left(
                                                                      \theta_c
                                                                      \Big(m, \frac{m}{r}\alpha     \Big)
                                                                      \right)
                                                      <
                                                      \frac{r}{m}
                                                      \tan 
                                                              \Big(
                                                              \theta_c(r, \alpha)
                                                              \Big) \, .
                           \end{equation}
                                                      \end{lem}

                                                      In fact a ratio type control on $\th_F$
                                                      can also be established. 
                               \begin{lem}\label{tF}
                                                      For $m,r\in \mathbb Z_+$  and $m>r$, it follows that
               \begin{equation}
                                 \label{ratiotF}
                                                      \tan \left(
                                                                      \th_F
                                                                      \Big(m, \frac{m}{r}\alpha     \Big)
                                                                      \right)
                                                      <
                                                       \frac{r}{m}
                                                      \tan 
                                                              \left(
                                                              \th_F
                                                                \left(
                                                                r, \alpha\sqrt{\frac{(m-1)r}{(r-1)m}}
                                                                \right)
                                                              \right) \, .
                         \end{equation}
                                                      \end{lem}
                                                      Although $F$ and $\th_F$ are more accurate and efficient for applying the criterion,
                                                      however in certain sense Lemma \ref{tF} is not as good as Lemma \ref{tc}
                                                      since the extra factor $\sqrt{\frac{(m-1)r}{(r-1)m}}$ is larger than one.
There will be no much difference in the scope of current paper when $m$ is sufficiently large and we save  its proof in appendix.

                                                      \begin{defn}
                                                      We call a regular minimal cone Type-c (or Type-F) Lawlor cone if it can be verified area-minimizing using the vanishing angle $\theta_c$ (or $\th_F$),
                                                      i.e., $\theta_c\leq \frac{R(L)}{2}$ (or $\th_F\leq \frac{R(L)}{2}$).
                                                      \end{defn}

                                                      We will show that the situations in Theorems \ref{main} and \ref{addthm} bring us  Type-c cones.
                                                      Besides, some considerations regarding $(\star)$ will be made in \S 4 (and results work for Type-F as well).

                 {\ }
                                                      
             \section{Minimal product}
                        The (constant) minimal product
                        \footnote{We say constant minimal product to distinguish this from spiral minimal product in \cite{L-Z}.}
                         in \S 1 can naturally induce a nice algorithm to generate minimal submanifold  for multiple minimal submanifolds 
                         \begin{equation}\label{prod}
         L_1\times\cdots
         \times L_n
         \triangleq
          \left(\lambda_1 L_1, \cdots, \lambda_n  L_n\right)\subset \mathbb S^{N_1+\cdots +N_n+n-1}.
       \end{equation}
         where $\lambda_i=\sqrt{\frac{k_i}{k}}$ with $k_i$ the dimension of $L_i$ and $k=\sum_{i=1}^n k_i>0$.
         In this paper we shall focus on the case that every $k_i>0$.
         
         There have been many specific situations where minimal product has been employed.
         For all $L_i$ being full spheres,
         the structure produces  Clifford type minimal surfaces.
          When  these $L_i$ are   R-spaces, \cite{OS} studied whether cones over  minimal products of them are area-minimizing.
          For general minimal submanifolds as inputs, see \cite{X} for a proof through the Takahashi Theorem.
         Around 2016 it was rediscovered independently by \cite{TZ} and \cite{CH} using different methods.
         Reader may choose either familiar way to understand the structure.
       
       \subsection{Formula of $\alpha^2$ for minimal product.\, }  \label{3.1}
         For simplicity, we mainly focus on  $n=2$ in the paper and all related results for $n\geq 3$ can be gained by induction.
          In \cite{TZ},  information about the second fundamental form of  minimal product has been clearly figured out.
           Let $k_1, k_2$ be the dimension of $L_1, L_2$ respectively.
            Take $x_1\in L_1$ and $x_2\in L_2$.
           Together with
         $
         \eta_0=(\la_2 x_1, -\lambda_1 x_2)
        $,
         orthonormal bases
         $\{\sigma_1,\cdots, \sigma_{N_1-k_1}\}$ and $\{\tau_1,\cdots,\tau_{N_2-k_2}\}$
         of $T_{x_1}^\perp L_1$ and $T_{x_2}^\perp L_2$ 
         lead to
          an orthonormal basis
         $$\{(\sigma_1,0),\cdots, (\sigma_{N_1-k_1},0),\; (0,\tau_1),\cdots, (0,\tau_{N_2-k_2}),\; \eta_0\}$$
         of the normal space of $L_1\times L_2$ at point $(\lambda_1 x_1,\la_2 x_2)$ in $\mathbb S^{N_1+N_2+1}$.

         Let $A$ stand for the symbol of shape operators.
         Then with respect to natural tangential orthonormal bases,
         we
         have the following relations:

         $ {\ \ \ \ \ \ \ \ \ \  \ \ } A_{(\sigma_i,0)} =
\begin{pmatrix}
\frac{1}{\lambda_1}A_{\sigma_i} & O\\
O &O
\end{pmatrix},
$
         \,
         $ A_{(0,\tau_j)} =
\begin{pmatrix}
O & O\\
O & \frac{1}{\la_2}A_{\tau_j}
\end{pmatrix},
$
\,
$
A_{\eta_0}=
                 \begin{pmatrix}
                 -\frac{\la_2}{\la_1}I_{k_1} & O\\
                 O & \frac{\la_1}{\la_2}I_{k_2}
                 \end{pmatrix}
                 .
                 $
Here $i=1,\cdots, N_1-k_1$ and $j=1,\cdots, N_2-k_2$.
If $N_1-k_1=0$ or $N_2-k_2=0$ or both vanish, then corresponding type would disappear.
For instance, for a Clifford type torus of codimension one, its second fundamental form is just the last matrix given by $\eta_0$.
        
            Let $\alpha_1^2,\, \alpha_2^2, \, \alpha^2$ be the supremum norm squares of corresponding second fundamental forms with respect to unit normals for $L_1,\ L_2$ and $L_1\times L_2$ respectively.
          Then  the following is obvius.
                        \begin{lem}\label{alpha}
                        As marked in the above, it follows that, for $k_1, k_2>0$,
                        $$\alpha^2=(k_1+k_2)\max\left\{ \frac{\alpha_1^2}{k_1},\, \frac{\alpha_2^2}{k_2},\, 1\right\}.$$
                        \end{lem}
                        \begin{proof}
                        Note that the matrices of shape operators with respect to three subclasses of unit normals $\{(\sigma_i,0)\}$,
                                  $\{(0,\tau_i)\}$
                                  and $\{\eta_0\}$ are mutually perpendicular.
                        \end{proof}

             \subsection{Formula of normal radius for minimal product}
             Another vital quantity  in Lawlor's criterion is the normal radius,
             so in this subsection we derive  its formula rigorously for the minimal product structure.
             
              Let $R_1,\, R_2, \, R$ be the normal radii of $L_1,\ L_2$ and $L_1\times L_2$ in corresponding Euclidean spheres respectively.
                                          For $(\lambda_1 x_1,\la_2 x_2)$ in $L_1\times L_2$, 
                             the combination 
                             \begin{equation}\label{nu}
                                                          \nu=a_1(\xi_1,0)+a_2(0, \xi_2)+a_0\eta_0
                            \end{equation}
                                     where $\xi_1$, $\xi_2$ are unit normals of $L_1$ and $L_2$ at $x_1$ and $x_2$ respectively,
                                     and
                                     $a_1^2+a_2^2+a_0^2=1$,
                                     can exhaust all unit normals of  $L_1\times L_2$ at $(\lambda_1 x_1,\la_2 x_2)$.
                                     Then, in $\mathbb S^{N_1+N_2+1}$,
                                     each unit normal induces a normal geodesic
          \begin{eqnarray}
                                     &&
                                     \exp_{(\lambda_1 x_1,\la_2 x_2)}(s\nu)
                                \label{exp}
                                     \\
                                    &=&
                                    (\cos s)(\lambda_1 x_1,\la_2 x_2 )+(\sin s)\nu 
                                    \nonumber
                                    \\
                           &=&
                           \big(
                           (\la_1\cos s + a_0\la_2\sin s )x_1
                           + 
                           (a_1 \sin s)   \xi_1   
                           ,
                              \nonumber\\
                   &&    \ \ \ \ \ \ \ \ \ \ \ \ \ \ \ \ \ \ \ \ \ \ \ \ \ \ \ \ \ \   \ \ \  \ \ \ \ \ \ \ \ \ \ \ \ \ \ \ \ \ \      (\la_2\cos s-a_0\la_1\sin s)x_2 
                               +     
                               (a_2 \sin s)   \xi_2          
                              \big)
                              \nonumber
     \end{eqnarray}
     
            Therefore, $R$ is the smallest positive number such that there exists 
            $$p\triangleq (\lambda_1 x_1,\la_2 x_2) \in L_1\times L_2$$
                                                           satisfying $ q\triangleq (\la_1 y_1, \la_2 y_2) = \exp_p(R\nu_p)\in L_1\times L_2$ for some $\nu_p\in S_p$.
                                                           Thus,
              \begin{eqnarray}
                     &( \la_1\cos R + a_0\la_2\sin R )x_1
                           + 
                           (a_1 \sin R)   \xi_1   
                           = 
                    \la_1\big(       (\cos \phi_1) x_1+ (\sin \phi_1) \xi_1
                    \big)
                    \label{e1}
\\
                 &  (\la_2\cos R-a_0\la_1\sin R)x_2 
                               +     
                               (a_2 \sin R)   \xi_2 
                        =
                        \la_2\big(  
                            (\cos \phi_2) x_2+ (\sin \phi_2) \xi_2
                                 \big)
                                 \label{e2}
              \end{eqnarray}
                   for some $\phi_1,\phi_2\in \R$ such that $$y_1=  (\cos \phi_1) x_1+ (\sin \phi_1) \xi_1 \text{\ \ \  and \ \ \ } y_2=  (\cos \phi_2) x_2+ (\sin \phi_2) \xi_2.$$
                   
                   Somehow similarly as argued in \cite{TZ},
                   there are three possibilities to realize normal radius $R$,
                   but notice that, a priori, possible
                   values of $\phi_1$ and $\phi_2$ in the above may not be discrete as in the isoparametric situation in \cite{TZ}.
                   Moreover, except being a point of sphere, the normal radius of a closed submanifold in sphere is no larger than $\pi$.
               
               
                   The formula of normal radius for minimal product can be derived.
                    \begin{lem}\label{Klem}
       Let $R_1$ and $ R_2$ be the normal radii of $L_1$ and $ L_2$ in $\mathbb S^{N_1}$ and $\mathbb S^{N_2}$ respectively.
       Then the normal radius $R\in (0,\pi)$ of $L_1\times L_2$ in $\mathbb S^{N_1+N_2+1}$ 
       satisfies 
     \begin{eqnarray}
           \cos R
           =
           1-\min\left\{\la_1^2(1-\cos R_1),\, \la_2^2(1-\cos R_2) \right\}.
           \label{l1}
   \end{eqnarray}
       \end{lem}
          \begin{rem}\label{conprod}
       Note that the result is in fact valid when $\la_1$ and $\la_2$ are just two nonzero real numbers $a_1, a_2$ with $a_1^2+a_2^2=1$. 
       However, we will restrict ourselves in the minimal product construction unless otherwise specified.
       \end{rem}
       
       \begin{rem}\label{Ri=pi}
       If some $L_i$ is the full sphere $\mathbb S^{N_i}$, then $R_i$ is defined to be $\pi$.
       \end{rem}
       
       \begin{proof}
       The three possibilities for $R$ based on \eqref{e1} and \eqref{e2} are the followings.
       
               {\bf Case I.} $y_1\neq x_1$ and $y_2\neq x_2$. 
                   
                    Now one must have $\cos \phi_1\leq \cos R_1$ and $\cos \phi_2\leq \cos R_2$ by the geometric meanings of $R_1$ and $R_2$.
                    The sum of $\big<\la_1x_1,\, $           \eqref{e1}$\big>$ and $\big<\la_2x_2 ,$ \eqref{e2}$\big>$ gives
        \begin{eqnarray}
                    \cos R=\la_1^2\cos \phi_1+\la_2^2\cos \phi_2
                             \leq
                             \la_1^2\cos R_1+\la_2^2\cos R_2
                                                              \label{e3}
     \end{eqnarray}

               {\bf Case II.} $y_1\neq  x_1$ and $y_2= x_2$ (or the other way around).
               
               Now $\cos \phi_2=1$ and \eqref{e3} becomes 
   \begin{eqnarray}
                    \cos R\leq \la_1^2\cos R_1+\la_2^2.
                                                              \label{e4}
     \end{eqnarray}
     
                  {\bf Case III.} $y_1= x_1$ and $y_2= x_2$.
       
           Since the ambient is Euclidean sphere, if $R$ occurs in this case then $R=2\pi$
                  (here $2\pi$ corresponds to normal geodesic curve for direction $\nu=\eta_0$ when neither $L_1$ nor $L_2$ contains antipodal points).
                  
                  {\ }
                  
                  So, to search for the value of $R$
                  one should focus on Case II
                  to see whether 
                    \begin{eqnarray}
                    \cos R= \la_1^2\cos R_1+\la_2^2.
                                                              \label{e4b}
     \end{eqnarray}
     can be reached.
                  Note that it is always true that
                  $$0<R<\min\{\la_1 R_1, \, \la_2 R_2\}<\pi.$$
                  Hence $\sin R\neq 0$.
                  So,
                  in Case II,
                  $a_2$ has to be zero in \eqref{e2}.
                  Now we want to solve:
  \begin{quote}
   Given $\phi_1=R_1$ for some $x_1$ and direction $\xi_1$ realizing normal radius $R_1$ of $L_1$ in $\mathbb S^{N_1}$, can we get $a_0,\, a_1$ (with $a_0^2+a_1^2=1$) and 
   $  \nu=a_1(\xi_1,0)+a_0\eta_0$
    to make \eqref{e1}, \eqref{e2} and \eqref{e4b} hold?
   \end{quote}
This can be done by
               \begin{eqnarray}
         & a_0
                     \sin R=\la_1\la_2(\cos R_1-1), 
         \label{c2}\\
        & a_1\sin R=\la_1 \sin R_1. \ \ \ \ \ \ \ \ \ \ \ \ \ 
         \nonumber
         \end{eqnarray}
  automatically with
  $a_0^2+a_1^2=1$.

       Therefore $\cos R$ takes the larger value of $\la_1^2\cos R_1+\la_2^2$ and $\la_1^2+\la_2^2\cos R_2$
       and \eqref{l1} follows for genuine submanifold inputs.
       
       Regarding Remark \ref{Ri=pi},
       when $L_1$ (or $L_2$ or both) is a full sphere, the normal radius $R_1$ (or $R_2$ or both) cannot be well defined as in Definition \ref{Dnr}.
For our purpose,  we always  define the normal radius of a full sphere $\mathbb S^{N_1}$ in itself to be $\pi$.
The reason is the following.
It can be realized as a totally geodesic sphere in  $\mathbb S^{N_1+\ell}$ for $\ell>0$ and now  $\mathbb S^{N_1}$ has normal radius $\pi$.
So the above argument works for the minimal product and the last $\ell$ components of $\R^{N_1+1+\ell}$ can be completely dropped as they contribute nothing 
         (%
         let $\xi$ be a unit normal vector to the minimal product at some point and if $\xi$ has a nonzero part coming from the last $\ell$ components of $\R^{N_1+1+\ell}$ then the geodesic with initial velocity vector $\xi$ in $\mathbb S^{N_1+\ell+N_2}$ will run at least length $\pi$ to return to the minimal product
         and hence this will not realize the normal radius as  the normal radius is strictly less than $\pi$ for the minimal product $-$ not a totally geodesic sphere, actually no more than $\frac{\pi}{2}$ by Lemma \ref{tcless90}).
            %
           %
           \footnote{
           For example, the normal radius of minimal product $\mathbb S^{N_1}\times \mathbb S^{N_2}$ in $\mathbb S^{N_1+N_2+1}$ is realized by $\eta_0$ and can reach $\frac{\pi}{2}$ only when $N_1=N_2$ according to  \eqref{e4b}, cf. the beautiful uniform formula in \cite{Law} for normal radius of minimal product among full spheres.
           }
              Thus, \eqref{l1} still holds with the understanding in Remark \ref{Ri=pi}.
       \end{proof}
                                                        \begin{rem}\label{indRk}
          Applying this to  $L_1\times L_2\times L_3=(L_1\times L_2)\times L_3$,  we have
           \begin{eqnarray}
          && \cos R(L_1\times L_2\times L_3) 
             \nonumber
\\
           &=&
           1-\min\left\{\frac{k_1+k_2}{k_1+k_2+k_3}(1-\cos R(L_1\times L_2)),\, \frac{k_3}{k_1+k_2+k_3}(1-\cos R_3) \right\}
                 \nonumber\\
                 &=&
                 1-\min\Bigg\{\frac{k_1}{k}(1-\cos R_1), \, \frac{k_2}{k}(1-\cos R_2),\,
        \frac{k_3}{k}(1-\cos R_3) 
                 \Bigg\}.
                       \nonumber
              \end{eqnarray}

                                                        \end{rem}
         By induction, one can get the following formula.
         \begin{cor}
         Let $R_i$ be the normal radii of $L_i$ in $\mathbb S^{N_i}$ for $i=1,\cdots, n$.
         Then the normal radius $R(L_1\times\cdots\times L_n)\in (0,\pi)$ of $L_1\times\cdots\times L_n$ in $\mathbb S^{N_1+\cdots +N_n+n-1}$  satisfies
         \begin{eqnarray}
           \cos R(L_1\times\cdots\times L_n)
           =
           1-\min_{1\leq i\leq n}\left\{\la_i^2(1-\cos R_i) \right\}.
           \label{l2}
   \end{eqnarray}
         \end{cor}
         \begin{rem}
        Similar statement formulated independently and some partial discussion included in \cite{JCX}.
        
         \end{rem}

{\ }

                    \section{Configuration results}
                                       
                                      This section 
                                      consists of three parts.
                                      
                                      In \S \ref{S41}
                                      the general configuration results Theorems \ref{main} and \ref{addthm}
                                      and Corollary \ref{cor}
                                      will be proved.
                                      The main idea is to 
                                      gain $\th_c< \frac{R}{2}$
                                      in the asymptotic sense
                                      based on Lemmas \ref{tc} and \ref{Klem}.
                                      
                                        In \S \ref{S42}
                                        some useful controls
                                        especially a key inequality on normal radius of minimal product $L_1\times L_2$
                                        will be given in Lemma \ref{KC}.
                                        Such an inequality 
                                        together with Lemma \ref{tc}
                                        can lead to $\th_c< \frac{ R}{2}$
                                        if both $\th_c$ and $R$ of $L_1\times L_2$ are decided by 
                                        the same component, which itself spans a Type-c cone.
                                       As a result, 
                                        $L_1\times L_2$ also spans a Type-c cone.
                                        
                                        In \S \ref{S43}
                                        we shall provide a partial positive answer to
                                        Question $(\star)$ in \S \ref{S1}
                                        under some assumptions.
                                        Besides, a quick proof of the main result of \cite{TZ}
                                        will be explained in Theorem \ref{TZmain} 
                                        following 
                                        Remark \ref{full}.
                                        Moreover, 
                                        we derive some more subtle structures 
                                        (e.g. dimension gap of Theorem \ref{main} in Theorem \ref{gap}) 
                                        complementing the asymptotic results in Theorems \ref{main} and \ref{addthm}.
                                        Some other configuration results within the category of Type-c cones will also be discussed.
                                     
                                     \subsection{Asymptotic results for general closed minimal submanifold.}  \label{S41}
                                       For simplicity, suppose that the set $\{L_i\}$ of closed minimal submanifolds in spheres  in assumption contains exactly one element $L_1^{k_1}\subset \mathbb S^{N_1}$.
                                       
                                  \begin{pfm}
                                  Let $R_1>0$ be the normal radius of $L_1$ in $\mathbb S^{N_1}$.
                                  Denote the minimal product of $n$-copies of $L_1$  by $L_1^{\times n}$.
                                  Then, by \eqref{l2} and Lemma \ref{tcless90} in below,
                                  the normal radius $R(L_1^{\times n})$ is no more than $\frac{\pi}{2}$ and its exact value is
            \begin{equation}\label{eqR}
 \arccos \left(1-\frac{1}{n}(1-\cos R_1)\right)
                                  =
                                  \sqrt {2(1-\cos R_1)}\cdot \frac{1}{\sqrt n} + \text{ higher order terms}.
             \end{equation}
                                  
                                  On the other hand, denote by $\alpha_1$ and $\alpha$ the supremum norms  of second fundamental forms of  $L_1$ and $L_1^{\times n}$ respectively.
                                  By induction, 
                                                   Lemma \ref{alpha}
                                                   implies that 
                                                   $$
                                                   \alpha=\sqrt{nk_1
                                                                       \max
                                                                       \left
                                                                       \{
                                                                       1,
                                                                       \frac{\alpha_1^2}{k_1}
                                                                      \right
                                                                       \}
                                                                       }
                                                                       =
                                                                       \frac{nk_1+1}{12}
                                                                        \frac{12}{nk_1+1}
                                                                       \sqrt{n{k_1}
                                                                       \max
                                                                       \left
                                                                       \{
                                                                       1,
                                                                       \frac{\alpha_1^2}{k_1}
                                                                      \right
                                                                       \}
                                                                       }.
                                                   $$
                                                   Now, by applying  Lemma \ref{tc} with $m=nk_1+1$ and $r=12$, we get
                                         \begin{eqnarray}
                                         \ \ \ \ \ \ \ \ \ \ \ \ 
                                                   \tan \theta_c
                                                   \Big(nk_1+1, \alpha
                                                   \Big)
                                                   <
                                                   \frac{ 12}{nk_1+1}\tan \theta_c
                                                            \left(
                                                                           12,
                                                                               \, 
                                                                                 \frac{12}{nk_1+1}
                                                                       \sqrt{n{k_1}
                                                                       \max
                                                                       \left
                                                                       \{
                                                                       1,
                                                                       \frac{\alpha_1^2}{k_1}
                                                                      \right
                                                                       \}
                                                                       }
                                                                                 \right)
                                                                                 \label{ineqtc}
                                     \end{eqnarray}
                                               when $nk_1+1>12$.
                                                   Note that, as $n\uparrow \infty$, the tangent term on the right hand side of \eqref{ineqtc} limits to $\tan 7.97^\circ$ according to the vanishing angle $\theta_c$ in the last column of the table on page 21 of \cite{Law}.
                                                   
                                                   Therefore, when $n$ is large enough, by comparing leading terms of \eqref{eqR} and \eqref{ineqtc} we have
                                                   $$
                                                   \theta_c
                                                   \left(nk_1+1, \alpha
                                                   \right)
                                                   <
                                                     \tan \theta_c
                                                   \big(nk_1+1, \alpha
                                                   \big)
                                                   <
\frac{R(L_1^{\times n})}{2}.
                                                   $$
According to Lawlor's criterion Theorem \ref{LC}, the cone $C(L_1^{\times n})$ is area-minimizing in $\R^{n(N_1+1)}$.                                  
                                  \end{pfm}     
                                  
                                       \begin{rem}\label{full}
                                       To get  full generality in Theorem \ref{main},
                                       notice that finitely many $\{L_i\}$
                                       also imply the finiteness property of 
                                       $ \max_{1\leq i\leq n}
                                                                       \left
                                                                       \{
                                                                       1,
                                                                       \frac{\alpha_i^2}{k_i}
                                                                      \right
                                                                       \}$
                                                                       and
                                                                       consequently
                                                                                  \begin{eqnarray}
                                         \ \ \ \ \ \ \ \ \ \ \ \ 
                                                   \tan \theta_c
                                                   \Big(k+1, \alpha
                                                   \Big)
                                                   <
                                                   \frac{ 12}{k+1}\tan \theta_c
                                                            \left(
                                                                           12,
                                                                               \, 
                                                                                 \frac{12}{k+1}
                                                                       \sqrt{k
                                                                      \max_{1\leq i\leq n}
                                                                       \left
                                                                       \{
                                                                       1,
                                                                       \frac{\alpha_i^2}{k_i}
                                                                      \right
                                                                       \}
                                                                       }
                                                                                 \right)
                                                                    \label{Rec}
                                     \end{eqnarray}
                                     where $k$ is the dimension of a minimal product among $\{L_i\}$.
                                                                       Also 
                                                                       note that
                                                                       \eqref{l2}
                                                                       leads to
                                                                       $
                                                                         \cos R
           \leq
           1-\frac{1}{k}\min_{1\leq i\leq n}\left\{        1-\cos R_i    \right\}
                                                                       $
                                                                       which gives
                     \begin{equation}\label{roughcontrol}
                                                                       R\geq
                                                                       \arccos
                                                                       \left(
                                                                       1-\frac{1}{k}\min_{1\leq i\leq n}\left\{        1-\cos R_i    \right\}
                                                                       \right)
                                                                        =
                                  \sqrt {2\min_{1\leq i\leq n}\left\{        1-\cos R_i    \right\}}\cdot \frac{1}{\sqrt k} + ...
                 \end{equation}
                             As a result, whenever $k$ is sufficiently large one can apply Theorem \ref{LC} for the conclusion of Theorem \ref{main}.
                             
                             If, further beyond finiteness, there are infinitely many $\{L_\ell\}_{\ell\in\mathscr A}$ with 
                       \begin{equation}\label{infmany}
                                    \sup_{\ell \in \mathscr A}
                                                                       \left
                                                                       \{
                                                                       1,
                                                                       \frac{\alpha_\ell^2}{k_\ell}
                                                                      \right
                                                                       \}
                                                                       <
                                                                       \infty
                                                                 \text{
                                                                 \ \ \ \ 
                                                                       and
                                                               \ \ \ \        }
                                                                       \inf_{\ell\in\mathscr A}     \left\{    k_\ell (1-\cos R_\ell )   \right\}
                                                                       >0,
                           \end{equation}
                                                                       then the conclusion of Theorem \ref{main} still holds.
                                                                                                                                        \end{rem}
                                       
                                       Note that, by \cite{L-Z}, for fixed $k_1\geq 3$ and every $\delta, \delta'>0$
                                       there exists spiral minimal product type $L_1$ of dimension $k_1$ in some sphere
                                       with its normal radius $R_1<\delta$ and $\frac{\alpha_1^2}{k_1}>(\delta')^{-1}$ where   $\alpha_1$ is its supremum norm  of  second fundamental form with respect to unit normals.
                                       So generally  there is no  control no these two quantities for a minimal submanifold (of codimension larger than two) in sphere.
                                       However, 
                                     machinery works regardless of 
                                     the initial figure of $L_1$ at all
                                     for
                                    spanning area-minimizing cones $C(L_1^{\times n})$ in the asymptotic sense in Theorem \ref{main}.
                                        
                                        {\ }
                                           
                                           Apparently, the idea here can lead to many general configuration results.
                                           Among others, we get a proof of Theorem \ref{addthm}.

                                        \begin{pfm3}
                                        Let $k_1$ be the dimension of $L\subset \mathbb S^{N}$.
                                              Then by Lemma \ref{Klem}
                                              the normal radius $R=R(L\times \mathbb S^d)$ of  $L\times \mathbb S^{d}$ 
                                                satisfies
                                                 $$
                                                       \cos R
           =
           1-\min\left\{\frac{k_1}{k_1+d}(1-\cos R_1),\, \frac{2d}{k_1+d} \right\}.
                                                 $$
                                           Hence, when $d$ is large (e.g. $d\geq k_1$),
                                           we get      
                                            $$
                                                       \cos R
           =
           1-\frac{k_1}{k_1+d}(1-\cos R_1)
           $$
                     Similar to the argument in the proof of Theorem \ref{main}
                     we have 
                     
                      \begin{equation*}\label{eqR'}
 \arccos \left(1-\frac{k_1}{k_1+d}(1-\cos R_1)\right)
                                  =
                                \frac{\sqrt {2k_1(1-\cos R_1)}}{\sqrt {k_1+d}} + \text{ higher order terms}.
             \end{equation*}

                                Now note that
                                by Lemma \ref{alpha}
                                 $$\alpha^2=(k_1+d)\max\left\{ \frac{\alpha_1^2}{k_1},\, 1\right\}.$$
                     Therefore, 
                     according to Lemma \ref{tc} with $m=k_1+d+1$ and $r=12$, we get
                                         \begin{eqnarray}
                              \nonumber
                                                   \tan \theta_c
                                                   \Big(k_1+d+1, \alpha
                                                   \Big)
                                                   <
                                                   \frac{ 12}{k_1+d+1}\tan \theta_c
                                                            \left(
                                                                           12,
                                                                               \, 
                                                                                 \frac{12}{k_1+d+1}
                                                                       \sqrt{(k_1+d)
                                                                       \max
                                                                       \left
                                                                       \{
                                                                       1,
                                                                       \frac{\alpha_1^2}{k_1}
                                                                      \right
                                                                       \}
                                                                       }
                                                                                 \right)
                                                                                 \label{ineqtc}
                                     \end{eqnarray}
                                               when $k_1+d+1>12$.
                     
                     So,  when $d$ is large enough, by comparing leading terms we have
                                                   $$
                                                   \theta_c
                                                   \left(k_1+d+1, \alpha
                                                   \right)
                                                   <
                                                     \tan \theta_c
                                                   \big(k_1+d+1, \alpha
                                                   \big)
                                                   <
\frac{R(L\times \mathbb S^d)}{2}.
                                                   $$
By Lawlor's criterion Theorem \ref{LC} the cone $C(L\times \mathbb S^d)$ is area-minimizing in $\R^{n+d+2}$.        

              {\ }
                     
                     If one replaces $\mathbb S^d$ by $\mathbb S^{d_1}\times \cdots \mathbb S^{d_\ell}$,
                     then the generalization in Remark \ref{rem14} holds  due to the same kind of reasoning.
                                                            \end{pfm3}

                                                             \begin{rem}
                        %
                               Given a closed minimal hypersurface $L$ in $\mathbb S^n$ ($n\geq 3$) and  any integer  $\tau\geq2$.
                              By  Theorem \ref{addthm} and Remark \ref{rem14}, 
                              it follows that based on $L$ there exist infinitely many area-minimizing cones in Euclidean spaces having the fixed codimension $\tau$.
                               \end{rem}

                                          Beside the advantage of control on the precise codimension as in Remark \ref{rem14},
                                          Theorem \ref{addthm} can lead to  a simple proof of Corollary \ref{cor}.

                                          \begin{pfm2}
                                          Set
                                          $$\mathscr A=\{\text{closed minimal submanifolds of dimension $> 0$ in Euclidean spheres}\}$$
                                          and
                                          $$\mathscr B=\{\text{regular area-minimizing cones of dimension $> 1$ in Euclidean spaces}\}.$$
                                          Then the mapping 
                                          $${\tt F}: B\in \mathscr B\longrightarrow B\cap \text{the corresponding unit sphere}\in \mathscr A$$
                                           is obviously injective. 
                                          
                                          Conversely, for every  $A(\subset \mathbb S^{N_A})$ $\in \mathscr A$, according to Theorem \ref{addthm} there exists a smallest integer $d_A\geq \dim A +1$  such that the cone over $A\times \mathbb S^d$ is area-minimizing when $d\geq d_A$
                                          \footnote{Also see Theorem \ref{gap} for a possible gap behavior.}.
                                          So, if we define 
                                          \begin{equation}\label{ttG}
                                          {\tt G}: A \longrightarrow C\left(A\times \mathbb S^{d_A+N_A} \right),
                                          \end{equation}
                                           then it is not hard to see that $\tt G$ is injective from $\mathscr A$ to $\mathscr B$.
                                          
                                          Hence, by the Schr\"oder–Bernstein theorem in the set theory, there exists a bijective map between $\mathscr A$ and $\mathscr B$.
                                          As a result,  they share the same cardinality. 
                                          \end{pfm2}
                                          
                                          \begin{rem}
                                          Here we add $N_A$ in the dimensional part of $\mathbb S$ in \eqref{ttG} to distinguish $A$ being a minimal submanifold in different ambient spheres,
                                          since we regard $A\subset \mathbb S^{N_A}$ and $A (\subset \mathbb S^{N_A})\subset \mathbb S^{N_A+1}$ as two distinct elements in $\mathscr A$.
                                          \end{rem}

                                          \begin{rem}
                                          By \cite{CM}, for any positive integer $\tt n$, there is a real $\tt n$-dimensional family of minimal Lagrangian tori in $\mathbb CP^2$ and therefore an $\tt n$-dimensional family of special Legendrian tori in $\mathbb S^5$.
                                          Hence the cardinality of special Legendrian tori in $\mathbb S^5$ is no less than that of the real numbers.
                                          \end{rem}
                                        
                                        \begin{rem}
                                        In fact one can also apply Theorem \ref{main} for Corollary \ref{cor} but needs to be more careful in choosing corresponding $\tt G$ map.
                                        \end{rem}
                                        
                                          {\ }
                                          
                                          Before closing this subsection,
                                          we would like to remark that all previously discovered area-minimizing cones $C(\tilde L)$ through Lawlor's curvature criterion share 
                                          the property (or restriction) that the curvature function $\alpha$ remains constant in the entire  $\tilde L$,
                                          whereas
                                          in our Theorems \ref{main} and \ref{addthm}
                                           $\alpha$ can be a non-constant function over $\tilde L$.

                                          {\ }

                                          
                                            \subsection{Key control on normal radius of minimal product construction. }\label{S42}
                                            Instead of asymptotic behaviors in \S \ref{S41},
                                            we shall focus on some control regarding Question $(\star)$ of \S \ref{S1}
                                            in this subsection.
                                           More precisely, we wish to show that, in some circumstances, 
                                           the minimal product of links $L_1, L_2$ of two Type-c cones 
                                           can again span a  Type-c cone. 
                                           So we aim to get the  inequality $\th_c(k+1,\alpha)<\frac{R(L_1\times L_2)}{2}$
                                           where $k=k_1+k_2$ and $k_i$ is the dimension of $L_i$ for $i=1,2$.
                                           
                                          In particular we will give some structural inequality for normal radius of minimal products.
                                          The first three results actually work for constant products, see Remark \ref{conprod}.
                                          Here the constant product construction merely means that $\{\la_i\}$ in the algorithm  \eqref{prod} are  nonzero numbers satisfying $\sum_{i=1}^n\la_i^2=1$.

                                           \begin{lem}\label{tcless90}
                                           Let $R_1, R_2, R$ be normal radii of $L_1, L_2$ and $L_1\times L_2$ respectively with $k_1, k_2\geq 1$.
                                           Then $R\leq \frac{\pi}{2}$.
                                           \end{lem}
                                           \begin{proof}
                                           Adding up the two possible values in \eqref{l1},
                                           we have 
                                \begin{equation}\label{accute}
                                           2\cos R\geq 1+\la_1^2\cos R_1+\la_2^2\cos R_2\geq 0.
                                  \end{equation}
                                           Hence $R\leq \frac{\pi}{2}$
                                           and 
                         equality holds exactly when $R_1=R_2=\pi$ and $k_1=k_2$ simultaneously.
                                           \end{proof}
                                           
                                           \begin{lem}\label{ineqcos}
                                           $
                                           \cos R> \cos R_i
                                        $
                                              \text{ and }
                                  $
                                           0<\sin R<\sin R_i
                              $
                                           \text{ for  } $i=1,2$.
                                           \end{lem}
                                           
                                           \begin{proof}
                                           Since $\cos R=\max\{\la_2^2+\la_1^2\cos R_1,\, \la_1^2+\la_2^2\cos R_2 \}\geq \cos R_i$ for $i=1,2$.
                                           As function $\cos(\cdot)$ is strictly decreasing in the interval $[0,\frac{\pi}{2}]$,
                                           by Lemma \ref{tcless90} we have $\sin R\leq \sin R_i$ for $i=1,2$.
                                           Since any closed submanifold has positive normal radius, we get $\sin R>0, \sin R_i>0$
                                           and all inequalities in the proof must be strict ones.
                                           \end{proof}

                                            The following is an immediate corollary.
                                           
                                           \begin{cor}\label{Dcoro}
                                           Assume that \eqref{l1} is taken for $i=1$, i.e., $\cos R=1-\la_1^2(1-\cos R_1)$.
                                           Then, we have
                    \begin{eqnarray}
                                           \tan \frac{R}{2}
                                           >
                                           \la_1^2\tan \frac{R_1}{2}.
                          \label{Dineq}
                           \end{eqnarray}
                                           \end{cor}
                                           
                                           \begin{proof}
                                           It follows from Lemma \ref{ineqcos} that
                     \begin{eqnarray}
                                    \tan \frac{R}{2}
                                           =
                                           \frac{1-\cos R}{\sin R}
                                         >
                                              \frac{\la_1^2(1-\cos R_1)}{\sin R_1}
                                              =
                                              \la_1^2
                                              \tan\frac{R_1}{2}.
                                                                                                           \nonumber
                           \end{eqnarray}
                                           \end{proof}

                                                          Here are some thoughts useful for our next step.
                                                                      Suppose that $L_1$ spans a Type-c cone.
                                                                      Then
                                                                   $$
                                           \tan \frac{R}{2}
                                           >   
                                           \la_1^2\tan \frac{R_1}{2}
                                           >
                                           \la_1^2\tan \th_c(k_1+1, \alpha_1),
                                           $$
                                           and we would be happy 
                                           if, for some reason,
                                           $$
                                          \tan \th_c(k_1+k_2+1, \alpha)
                                          \leq 
                                           \frac{k_1}{k_1+k_2}
                                           \tan \th_c(k_1+1, \alpha_1),
                                           $$
                                           since
                                           this could imply
                                           that 
                                           $\th_c(k+1,\alpha)<\frac{R(L_1\times L_2)}{2}$
                                           as we want.
                                           But,
                                           comparing this with 
                                           Lemma \ref{tc},
                                           one can see that the ratio outside $\tan$
                                           is $\frac{k_1+1}{k_1+k_2+1}$
                                           which is larger than 
                                           the constant 
                                                                                      $\frac{k_1}{k_1+k_2}$.
                                                         It seems very hard or likely impossible to
                                                         improve the ratio
                                                         on the right hand side
                                                         of \eqref{ratiotc}
                                                         in Lemma \ref{tc}
                                                         to be 
                                                         $\frac{k_1}{k_1+k_2}$.
                                                         (The same problem exists for the Type-F estimate \eqref{ratiotF} of Lemma \ref{tF}.)
                                        To conquer this difficulty,
                                        instead we turn to improve the inequality \eqref{Dineq}
                                        and get what we need for our purpose.  

                                        \begin{lem}[Key control]
                                        \label{KC}
                                        As in Corollary \ref{Dcoro},
                                        $\cos R=1-\la_1^2(1-\cos R_1)$.
                                   Assume 
                                   
                                    (1) $k_2\geq 4$ \ \ or\ \  
                                    (2) $k_2\geq 3, k_1\geq 2$ 
                                      \\  and assume $R_1\leq \frac{\pi}{2}$.
                                        Then for the minimal product construction we have
                                             \begin{eqnarray}
                                           \tan \frac{R}{2}
                                      \geq 
                                          \frac{k_1+1}{k_1+k_2+1}
                                           \tan \frac{R_1}{2}.
                          \label{Impineq}
                           \end{eqnarray}
                                        \end{lem}

                        \begin{proof}
                        With $k=k_1+k_2$,
                        the assumption tells that
                         $
                          \cos^2 R
                               =
                                      \left(
                                              \frac{k_2+k_1\cos R_1
                                              }
                                              {k}
                                              \right)^2
                                                         $.
                        Hence,
                              \begin{eqnarray}
                                    k^2\sin^2 R
                                         &  
                                         =
                                         &
                                         k^2-k_2^2
                                                                        -2k_1k_2\cos R_1
                                                                        -k_1^2\cos^2 R_1
                   \nonumber\\
                                           &
                                           =
                                           &                             
                                      k^2-k_2^2-k_1^2
                                                                        -2k_1k_2\cos R_1
                                                                        +k_1^2\sin^2 R_1
                       \nonumber\\
                                           &
                                           =
                                           &
                                           2k_1k_2(1-\cos R_1)   
                                                     +k_1^2\sin^2 R_1
                   \nonumber\\
             \left(\text{by  assumption } 0<R_1\leq \frac{\pi}{2}\right)            \ \ \ \ \ \             &
                                     \leq
                                     &
                                      2k_1k_2(1-\cos R_1)(1+\cos R_1)   
                                                     +k_1^2\sin^2 R_1
                        \nonumber\\
                                   &
                                   =
                                   &
                                        k_1 (2k_2+k_1)
                                         \sin^2 R_1   
                                                                                                        \nonumber
                           \end{eqnarray}
                       and thus   
                 $
                                   \sin R
                                             \leq 
                                                          \frac{
                                                          \sqrt{k_1(2k_2+k_1)}
                                                          }
                                                          { k}
                                                          \sin R_1.
              $

                  Following the argument of Corollary \ref{Dcoro} but now with 
                             $
                                   \sin R
                                             \leq 
                                                          \frac{
                                                          \sqrt{k_1(2k_2+k_1)}
                                                          }
                                                          { k}
                                                          \sin R_1
              $,
                          in order for
                          $$
                                           \tan \frac{R}{2}
                                   \geq
                                            \frac
                                                                                                      { k_1}  {
                                                          \sqrt{k_1(2k_2+k_1)}
                                                          }
                                         \tan \frac{R_1}{2}
                                     \geq
                                       \frac{k_1+1}{k+1}
                                          \tan \frac{R_1}{2},
                                           $$
                                it requires either $k_2\geq 4$ or            
                                $k_2\geq 3, k_1\geq 2$.
                        \end{proof}

                           \begin{rem}\label{Mcosineq}
                           Similarly as in Remark \ref{indRk},
                            under the assumption that 
                            $\cos R(L_1\times\cdots\times L_n)=1-\la_1^2(1-\cos R_1)$ and   $\sum_{i=2}^n k_i\geq 4$ and $R_1\leq \frac{\pi}{2}$,
                            one can extend the result to the situation of multiple inputs 
                                                         \begin{eqnarray}
                                           \tan \frac{R(L_1\times\cdots\times L_n)}{2}
                                           \geq 
                                          \frac{k_1+1}{\sum_{i=1}^n k_i+1}
                                           \tan \frac{R_1}{2}.
                          \label{ImpM}
                           \end{eqnarray}
                     Here we think of $L_1\times \cdots \times L_n$ as $L_1\times (L_2\times\cdots \times L_n)$.
                           \end{rem}

                           \begin{rem}
Note that by Lemma \ref{tcless90} every minimal product type $L_1$ automatically satisfies  the requirement $R_1\leq \frac{\pi}{2}$.
                           \end{rem}

                       {\ }    
                           
                           With the preparation of this subsection, we are ready to derive more configuration results in the next subsection.
                           
                            \subsection{More configuration results.}\label{S43}
                            For better exposition,
                            let us  introduce some useful concepts.

                            \begin{defn}
                           Let $\alpha$ be the supremum norm square of  second fundamental form with respect to unit normals for closed minimal submanifold $L$  of dimension $k$ in sphere.
                           We define the quantity $\frac{\alpha^2}{k}$ to be the slope of $L$, denoted by {\tt s}$(L)$.
                           Say $L$ (and also C(L)) is of class (a), (b), (c) if  {\tt s}$(L)<1$,  {\tt s}$(L)=1$ and  {\tt s}$(L)>1$ respectively.
                           The union of classes (a) and (b) is called class (I), and that for (b) and (c) is called class (II).
                            \end{defn}

                           There are plentiful examples for each type.
                                                      For some neat ones,
                           we may consider isoparametric foliations of spheres.
                           Here we mention some brief information about  their {\tt s}$(L)$.
                           Interested readers may  
                           check  \cite{TZ} and references therein for further knowledge.
                           
                           A closed embedded  hypersurface in a unit sphere
                           is called isoparametric if it has constant principal curvatures.
                           This property is inherited by its parallel hypersurfaces.
                           In such way it decides an isoparametric foliation with two focal submanifolds $M_+$ and $M_-$.
                           Let {\tt g} be the number of distinct principal curvatures, $\cot \theta_\beta$ $(\beta=1,\cdots, g; 0<\th_1<\cdots<\th_{\tt g}<\pi)$ of a regular isoparametric leaf and $m_\beta$ the multiplicity of $\cos \th_\beta$.
                           M\"unzner showed that {\tt g} must be 1,2,3,4 or 6 and $m_\beta=m_{\beta+2}$
                           with indices mod {\tt g}.
                           Distributions of $({\tt g}, m_1, m_2)$ are completely clear.
                           The slopes for the unique minimal isoparametric hypersurface and two focal submanifolds in an isomparametric foliation of sphere 
                           only depend on the distribution $({\tt g}, m_1, m_2)$ regardless of a  complete classification of isoparametric foliations of spheres.
                           
                           {\ }\\
                           {\bf Example 1. }
                           Let $M$ be the  minimal isoparametric hypersurface in an isomparametric foliation of sphere with {\tt g}.
                           Then {\tt s}$(M)={\tt g-1}$.
                           
                           Here {\tt g=1} for totally geodesic hyperspheres with normal radius $\pi$, 
                          which belong to class (a);
                                         {\tt g=2} for Clifford type minimal hypersurfaces,
                                         which belong to class (b);
                           {\tt g=3} has four examples with $m_1=m_2=m_3=1,2,4,8$,
                           which belong to class (c);
                            for {\tt g=6} there are only two values $m_1=m_2=1,2$ and corresponding hypersurfaces are of class (c).
                            When {\tt g=4}
                            there are infinitely many geometrically distinct examples of minimal hypersurfaces and all of them are of class (c).
                           
                           {\ }
                           \\
                              {\bf Example 2. } Focal submanifolds.
                               For any isoparametric foliation with {\tt g=4},
                              note that $M_+$ has dimension $m_1+2m_2$ and $M_-$ has dimension $2m_1+m_2$ 
                              with $\alpha_+^2=2m_2$ and $\alpha_-^2=2m_1$ respectively.
                              Hence,  both are  of class (a). 
                           
                           For {\tt g=3} and $m_1=m_2=m_3=1,2,4$ or $8$, 
                          each focal submanifold $M_\pm$  has {\tt s}$(M_\pm)=\frac{1}{3}$, and thus they are of class (a);
                            for {\tt g=6} and $m_1=m_2=1$ or $2$, 
                          each focal submanifold $M_\pm$  has {\tt s}$(M_\pm)=\frac{4}{3}$, and thus they are of class (c).
                           
                           Note that each focal submanifolds has normal radius exactly $\frac{2\pi}{\tt g}$.

                           {\ }
                           
                           Here are some simple properties of these classes, according to Lemma \ref{alpha}.

                           \begin{prop}\label{p1}
                           Minimal product $L_1\times L_2$ of two minimal submanifolds $L_1, L_2$ in spheres  belongs to either class of (b) or (c).
                           \end{prop}

                           \begin{prop}
                           Given $L_1, L_2$ of class (I).
                          Then $L_1\times L_2$ belongs to class (b).
                           \end{prop}
                           
                            \begin{prop}
                          If either $L_1$ or $L_2$ is of class (c),
                          then $L_1\times L_2$ belongs to class (c).
                          If either $L_1$ or $L_2$ is of class (II),
                          then $L_1\times L_2$ belongs to class (II).
                           \end{prop}
                           
                           {\ }
                           \\
                               {\bf Example 3. } 
                               Lawlor classified all area-minimizing cones over links which are minimal products of spheres.
                               All these links are  of class (b).
                            
                            {\ }\\
                               {\bf Example 4. }
                               Our Theorem \ref{main} can manufacture a constellation  of elements of class (b) and class (c). 
                               For example, given finitely many closed minimal submanifolds $\{L_i\}$ in spheres,
                               whenever they are not much curved, i.e., $\alpha_i^2\leq k_i$,
                               our Theorem \ref{main} can apply and produce lots of members in class (b).
                               
                               Moreover, 
                               to the countably many  $\{L_\ell\}_{\ell \in\mathscr A}$ of focal submanifolds with ${\tt g}=4$ in Example 2 above
                               which satisfy \eqref{infmany},
                               we can apply Remark \ref{full}.
                      
                      {\ }
                        
                          In fact as an application we can cover the main result of \cite{TZ} based on the structural discovery found here.
                          For doing so let us recall the following basic property.
                               
                               \begin{prop}\label{piso}
                               Let $L_1^{k_1}$ be either a minimal isoparametric hypersurface or a focal submanifold
                               and $R_1$ its normal radius.
                               Then it follows that 
             \begin{equation}\label{lbcontrol}
             k_1(1-\cos R_1)> 0.8
                               \text{
                              \ \ \  and \ \ \ }
                               \alpha_1^2\leq 5k_1.
                 \end{equation}               
                               \end{prop}
                               
                               \begin{proof}
                               When $L_1$ is a minimal isoparametric hypersurface of a sphere with ${\tt g}$,
                               the value of $\cos R_1$ is clear 
                               and the author computed that in (56) of \cite{TZ}
                               :
                                   \begin{equation}\label{cosphi}
         \,
   \cos R_1
   =\begin{cases}
  -1 &\ {\tt g}=1,\\
1-\frac{2\min\{m_1,\;m_2\}}{m_1+m_2}  &\ {\tt g}=2,\\
 \frac{1}{2} &\ {\tt g}=3,\\
 \sqrt{1-\frac{\min\{m_1,\;m_2\}}{m_1+m_2}} &\ {\tt g}=4,\\
\frac{\sqrt 3}{2} &\ {\tt g}=6.
   \end{cases}             
\end{equation}
                             Note that  we have now $k_1=m_1+m_2$ for ${\tt g}=2$;
                                                                     $k_1=3m_1$ for ${\tt g}=3$;
                                                                     $k_1=2(m_1+m_2)$ for ${\tt g}=4$;
                                                                     $k_1=6m_1$ for ${\tt g}=6$.
                                    So, 
                                                                 \begin{equation}\label{cosphi}
         \,
  k_1(1- \cos R_1)
   =\begin{cases}
  2k_1 &\ {\tt g}=1,\\
2\min\{m_1,\;m_2\}  &\ {\tt g}=2,\\
 \frac{3}{2}m_1 &\ {\tt g}=3,\\
 2(m_1+m_2)\left(1-\sqrt{1-\frac{\min\{m_1,\;m_2\}}{m_1+m_2}}\right) &\ {\tt g}=4,\\
6m_1\left(1-\frac{\sqrt 3}{2}\right) &\ {\tt g}=6.
   \end{cases}             
\end{equation}
                            The second last is larger than $\min\{m_1,\;m_2\}$
                            and the last is bigger than $0.8038m_1$.
                            
                      When $L_1$ is a focal submanifold of some isoparametric foliation of sphere,
                      as already mentioned  $R_1=\frac{2\pi}{\tt g}$.
                      Note that 
                      we have 
                      $k_1=0$ for ${\tt g}=1$ (not in our consideration with positive $k_1$);
                      $k_1=m_1$ or $m_2$ for ${\tt g}=2$;
                      $k_1=2m_1$ for ${\tt g}=3$;
                      $k_1=2m_1+m_2$ or $m_1+2m_2$ for ${\tt g}=4$;
                                                  $k_1=5m_1$  for ${\tt g}=6$.
                                                  So, 
                                                  $k_1(1-\cos R_1)\geq 1$.
                                                  
                                                  As a result, the quantity $k_1(1-\cos R_1)\geq 0.8$ as stated.
                                                 
                                                  As for $\alpha_1^2$,
                                                  it was also computed in \cite{TZ} 
                                                  with the conclusion $\alpha_1^2=({\tt g}-1)k_1$ for the minimal isoparametric hypersurfaces;
                                                  $\alpha_1^2=0$ for focal submanifolds with ${\tt g}=2$;
                                                   $\alpha_1^2=\frac{k_1}{3}$ for focal submanifolds with ${\tt g}=3$;
                                                  $\alpha_1^2=2m_2$ or $2m_1 \leq k_1-1$ for focal submanifolds with ${\tt g}=4$;  
                                                   $\alpha_1^2=\frac{4k_1}{3}$ for focal submanifolds with ${\tt g}=6$.
                                                   Hence,  $\alpha_1^2\leq 5k_1$.
                              \end{proof}

            \begin{thm}\label{TZmain}
            Let  $\{L_\ell\}_{\ell \in\mathscr A}$ in Remark \ref{full} be the collection of all minimal isoparametric hypersurfaces and focal submanifolds of spheres.
            Then cones, of dimension no less than $37$, over minimal products among the collection are all area-minimizing.
            \end{thm}
            
            \begin{proof}
            Use $L$ to denote the minimal product of dimension $k> 11$ with normal radius $R$.
            As
                                     \begin{eqnarray}                             
                                                                                                     \tan \theta_c
                                                   \Big(k+1, \alpha
                                                   \Big)
&\leq&
                                                   \frac{ 12}{k+1}\tan \theta_c
                                                            \left(
                                                                           12,
                                                                               \, 
                                                                                 \frac{12}{k+1}
                                                                       \sqrt{5k
                                                                       }
                                                                                 \right)
                                                 \nonumber
                                                                                 \\
                           \text{ when $k\geq 36$}    \ \ \ \ \ \ \ \ \ \ \ \ \ \ \ \ \ \  \ \ \ \ \  \ \ \ \
                                            & <&
                                                                                     \frac{ 12}{k+1}
                                                                                     \tan \theta_c
                                                            \left(
                                                                           12,
                                                                               \, 
                                                                                \sqrt {19}
                                                                                 \right)
                              \label{36}                  \\
                                                                          &       <&
                                                                                     \frac{ 12    }{k+1}
                                                                                     \tan 11.23^\circ         
\nonumber
\\
                                   & <&
                                      \frac{ 2.383        }{k+1}
                                      \nonumber
             \end{eqnarray}
             and
             by Proposition \ref{piso}
             $$
                  R>
                                                                       \arccos
                                                                       \left(
                                                                       1-\frac{0.8}{k}
                                                                       \right)
                                                                        =
                               \sqrt{  \frac{1.6}{ k} } + \text{ positive terms}
                                  >
                                \sqrt{  \frac{1.6}{ k} }\, ,
             $$
             now let us  solve the sufficient inequality to apply Lawlor's criterion
             $$
             \th_0(C(L))
             <
             \th_c(k+1, \alpha)
             <
             \tan \theta_c
                                                   \Big(k+1, \alpha
                                                   \Big)
                                                   <
                                                   \frac{2.383 }{k+1}
                                                   <\frac{\sqrt{1.6}}{2\sqrt {k}}
                                                   <\frac{R}{2}
                                                  \,  .
             $$
             It suffices to have 
             $
             14.2*k<(k+1)^2
             $
            and this is always true  for $k\geq 13$.
            Therefore, the theorem gets proved by the requirement \eqref{36}.
            \end{proof}
                                {\ }
                               
                   Now let us  prove a useful lemma on the existence of vanishing angle. 
                               \begin{lem}[Existence of $\th_c$]\label{ThcE}
                               Suppose $L_1, L_2$ span Type-c cones 
                           with  $k_1+k_2\geq 11$.
                           Then, 
                           for $\alpha=\alpha (L_1\times L_2)$,
                           the vanishing angle $\th_c(k_1+k_2+1, \alpha)$
                           exists.
                               \end{lem}
                            
                            \begin{proof}
                            When $\alpha^2=k=k_1+k_2$,
                            by Lemma \ref{tc} one can see that
              \begin{equation}\label{cmF0}
                                                   \tan \theta_c
                                                   \Big(k+1, \sqrt{k}
                                                                                                      \Big)
                                                   \leq
                                                   \frac{ 12}{k+1}\tan \theta_c
                                                            \left(
                                                                           12,
                                                                               \, 
                                                                                 \frac{12}{k+1}
                                                                       \sqrt{k
                                                                       }
                                                                                 \right)
                                                                                 <
                                                                                   \frac{ 12}{k+1}\tan \theta_c
                                                            \left(
                                                                           12,
                                                                                 \sqrt{12}                                                                                 \right)
                                     \end{equation}
                                     \text{exists}.

                                     When $\alpha^2=\frac{k\alpha_\tau^2}{k_\tau}$ either for $\tau=1$ or for $\tau=2$,
                \begin{eqnarray}
                                                 &&  \tan \theta_c
                                                   \Big(k+1, \alpha
                                                                                                      \Big)
                                                                                                      \nonumber\\
                                                   &<&
                                                   \frac{ k_\tau+1}{k+1}
                                                   \tan \theta_c
                                                            \left(
                                                                           k_\tau+1,
                                                                               \, 
                                                                              \frac{ k_\tau+1}{k+1}
                                                                  \alpha
                                                                                 \right)
                                                                       \nonumber          \label{cpF1}\\
                                                                         &   =&
                                                                              \frac{ k_\tau+1}{k+1}
                                                                                 \tan \theta_c
                                                            \left(
                                                                          k_\tau+1,
                                                                       \frac{ \sqrt{k_\tau}+\frac{1}{\sqrt{k_\tau}}}{\sqrt{k}+\frac{1}{\sqrt{k}}}       \alpha_\tau                                                                      \right)
                                                       \nonumber             \label{cpF2}   \\
                                                                     &   <&
                                                                              \frac{ k_\tau+1}{k+1}
                                                                                 \tan \theta_c
                                                            \left(
                                                                          k_\tau+1,
                                                               \alpha_\tau                                                                      \right)
\label{cpF3}
                 \end{eqnarray}
and thus $\theta_c
                                                  (k+1, \alpha
                                                                                                      )$
                                                                                                      exists  by the assumption  $C(L_\tau)$ is a Type-c cone.

  Thus, no matter which value  $\alpha^2$ takes from $\left\{k, k\frac{\alpha^2_1}{k_1}, k\frac{\alpha^2_2}{k_2}\right\}$,
                   we see from the above that
                   $\theta_c
                                                   (k+1, \alpha
                                                                                                      )$
                                                                                                      always exists.
                            \end{proof}
                            
                    {\ }
                            
                           We are ready to obtain more configuration results in this subsection.
                           
                           \begin{thm}\label{t2}
                           Suppose $L_1, L_2$ span Type-c cones 
                           with $k_1,\, k_2\geq 3$. 
                           If {\tt s}$(L_1)$={\tt s}$(L_2)\geq 1$,
                           then their minimal product $L_1\times L_2$ spans a Type-c cone of class (II).
                           \end{thm}
                        
                           
                           \begin{proof}
                           Without loss of generality,
                           assume that 
                                                                   $\cos R=1-\la_1^2(1-\cos R_1)$.
                                                                   There are two cases to be discussed separately.
                                                               
                                                               {\ }
                                                                  
                                                                   Case 1. $R_1\leq \frac{\pi}{2}$.
In this case we     can                                                                                             
apply the key control Lemma \ref{KC}
under the assumption on dimensions
($k_1,\, k_2\geq 3$ would be enough for \eqref{dimensions}, as condition (2) of Lemma \ref{KC} satisfied)
                to get 
                   \begin{eqnarray}
                                           \tan \frac{R}{2}
                                           &\geq&
                                          \frac{k_1+1}{k_1+k_2+1}
                                           \tan \frac{R_1}{2}
                     \label{dimensions}
                     \\
                                           &\geq&
                                           \frac{k_1+1}{k_1+k_2+1}
                                           \tan \big(\th_c(k_1+1, \alpha_1)\big)
                                           \nonumber\\
                                          & >&
                                          \tan 
                                          \big(\th_c(k_1+k_2+1,  \frac{k_1+k_2+1}{k_1+1}\sqrt{{\tt s}(L_1)k_1})\big)
                                          \nonumber
                                          \\
                                          & =&
                                           \tan 
                                          \big(\th_c(k_1+k_2+1,  \frac{\sqrt{k_1+k_2}+\frac{1}{\sqrt{k_1+k_2}}}{\sqrt {k_1}+\frac{1}{\sqrt {k_1}}}\sqrt{{\tt s}(L_1)(k_1+k_2)})\big)
                                           \nonumber\\
                                          & >&
                                            \tan 
                                          \big(\th_c(k_1+k_2+1,  \alpha) \big)
\label{tobealpha}
%
                           \end{eqnarray}
                                The last inequality \eqref{tobealpha} is due to the increasing property of $x+\frac{1}{x}$ when $x\geq 1$,
                                the increasing property of $\th_c(m, \cdot)$ in the second slot
                                \footnote{\label{ft}This is clear in \cite{Law} but to be self-contained we give a short proof in appendix.},
                                and the fact that $\alpha^2=(k_1+k_2){\tt s}(L_1)$          
                                under the assumption.
                                
                                   {\ }
                                   
                                                                   Case 2. $\frac{\pi}{2}< R_1\leq \pi$.
                 By the proof of Lemma \ref{ThcE} for {\tt s}$(L_1)$={\tt s}$(L_2)\geq 1$ the vanishing angle $\theta_c
                                                   (k+1, \alpha
                                                                                                      )$
                                                                                                      always exists
                                                                                                      without  bothering
                                             \eqref{cmF0}  
                                              so that we   have no assumption on total dimension.                     
                                                                                                      Moreover, 
                                                                                                      there is a free rough control
                                                                                                        $
                                                                                                        \theta_c
                                                   (k+1, \alpha
                                                                                                      )
                                                                                                      \leq
                                                                                                        \tan \theta_c
                                                   (k+1, \alpha
                                                                                                      )
                                                                                                      <\frac{1}{\alpha}$
                                                                                                      since the function $(1-\alpha t)e^{\alpha t}$ vanishes at $t=\frac{1}{\alpha}$ (see appendix for more information about the equivalent calibration inequality where  solution needs to reach zero before $t=\frac{1}{\alpha}$).
                                                 
                                                 Under the assumption of Case 2, 
                                                 we have 
                                                 $\cos R\leq 1-\la_1^2$
                                                 which  implies
                                                 $$
                                                 2
                                                 \left(\frac{R}{2}\right)^2\geq
                                                  2\sin^2{\frac{R}{2}}
                                                  =1-\cos R
                                                  \geq 
                                                  \la_1^2,
                                                 $$     
                                                 and hence     
                                                 $\frac{R}{2}\geq \sqrt{\frac{k_1}{2k}}$.      
                                                 As $\alpha\geq \sqrt k$ ,
                                                 we          also get           
                                                 $
                                                 \frac{R}{2}\geq \theta_c
                                                   (k+1, \alpha
                                                                                                      )$            
                                                                                                      when $k_1\geq 2$.  
                                                                                                      
                                                                                                      {\ }
                                                                                                      
                                                                                                      Therefore, in both Case 1 and Case 2,
                                                                                                      one can apply Lawlor's criterion
                                                                                                      and
                                                                                                      the proof gets completed.
                           \end{proof}

                          \begin{rem}\label{classb}
                       Examples of class (b) are many, see Example 4.    
                       All elements of class (b) has slope one.
                       This theorem 
                       shows its power
                       and derives
                       a universal structure for the category of Type-c cones of class (b).
                          \end{rem}
                          
                           Not as in Theorem \ref{main} where sufficiently many copies of a given $L$ are needed,
                           for links of Type-c cones of class (II) two copies can work.
                           
                           \begin{cor}\label{t2c}
                           Let $L^{k}$ be the link of a Type-c cone of class (II) with $k\geq 3$ and $n\geq 2$.
                           Then the minimal product $L^{\times n}$ can also  span a Type-c cone of class (II). 
                           \end{cor}

                               From the proof of Theorem \ref{t2},
                           we know that as long as $\cos R$ and $\alpha^2$ are determined by the same slot of the minimal product
                           then the minimal product spans a Type-c cone.
                           There are many specific cases for which this  coincidence in the same slot occurs.
                           A general  structure result is the following (cf. Theorem \ref{addthm}).
                           
                           \begin{thm}\label{t3}
                           Let $L^k$ be the link of a Type-c cone  of class (II)  and $n\geq k\geq 3$.
                           Then $L\times \mathbb S^n$ can span a Type-c cone.
                           \end{thm}

                           \begin{proof}

                           Under the dimension assumption,
                           we can see 
                           from 
                           Lemma \ref{Klem}
                           that
                           $$\cos R=1-\frac{k}{n+k}(1-\cos R(L)).$$
                           Moreover, by Lemma \ref{alpha}
                           we have 
                           $\alpha^2(L\times \mathbb S^n)=(k+n)\max\{1, {\tt s}(L)\}=(k+n){\tt s}(L)$.
                           Consequently, following the reasoning in the preceding paragraph, the conclusion stands true.
                           \end{proof}
                           
                           Moreover, a subtle structure result  related to Theorem \ref{main} can be obtained.
                           
                           \begin{thm}\label{gap}
                           Suppose $L=L_1^{\times s_1} \times \cdots \times L_n^{\times s_n}$ with $n\geq 2$, $s_i\geq 1$ 
                           and each component dimension $k_i>0$ for $i=1,\cdots, n$ spans a Type-c cone.
                           Then so does 
                           $
                           \tilde L\triangleq 
                           L_1^{\times (s_1+\ell_1)} \times \cdots \times L_n^{\times (s_n+\ell_n)}$ where $\ell_i\geq 0$ and $\sum_{i+1}^n k_i\ell_i\geq 3$.
                           \end{thm}
                          
                          \begin{rem}
                          Here  $L_i$ itself may not span area-minimizing cone.
                          \end{rem}
                          
                          \begin{rem}
                          If we use $L\prec \check L$ to describe 
                          the relation of increasing exponential order between $L$ and $\check L$,
                          then,
                          given a strictly increasing sequence of this relation,
                           according to Remark \ref{full}
                           there must be a (largest) successive infinite tail 
                           elements of which all span area-minimizing cones.
                           The result here
                           indicates
                           that the tail
                           and the element in the sequence  which spans an area-minimizing cone of lowest dimension
                           has gap at most $2$ in dimension.
                          \end{rem}
                          
                          
                          \begin{proof}
                          Without loss of generality, we consider $\check L \triangleq L_1^{\times (s_1+1)} \times \cdots \times L_n^{\times s_n}$ with $k_1\geq 3$ only
                          and other situations can be derived accordingly.
                         Note that by Lemma \ref{tcless90} it follows $R(L)\leq \frac{\pi}{2}$. 
                         By Proposition \ref{p1} or Lemma \ref{alpha} we have ${\tt s}(L)\geq \max\{1, {\tt s}(L_1)\}$.
                         Therefore, 
                         both
                         $R(\check L)$ and ${\tt s}(\check L)$
                         are decided by the first input $L$ of 
                         the minimal product of
                         $
                         L
                         $
                         and $L_1$.
                         When $k_1\geq 3$,
                         the key control Lemma \ref{KC} applies
                         and the argument in the proof in Case 1 of Theorem \ref{t2}
                         leads us to the conclusion of the theorem.
                          \end{proof}
                          
                           In fact more subtle structures can be established.
                           To illustrate the point, let us mention two more situations,
                           based on which interested readers may figure out many others.
  
                           \begin{thm}\label{manyS}
                           Let $L^k$ be the link of a Type-c cone with $k\geq 3$.
                           If 
                           %
                           $L$ is of class (II) 
                         %
                          then
                            $L^k\times \mathbb S^{d_1}\times\cdots \mathbb S^{d_n}$  spans a Type-c cone.
                           \end{thm}
                           \begin{proof}
                            The reason is following.
                           If some slot $ \mathbb S^{d_\tau}$ decides $R$ instead of component $L$,
                           then we are in Case 2 in the proof of Theorem \ref{t2}
                           and can get  $\frac{R}{2}\geq\sqrt{\frac{d_\tau}{d}}$  now and 
                           $\th_c(d+1, \alpha)
                                      <
                                            \frac{1}{\alpha}
                                            \leq \sqrt{\frac{1}{d}}
                                                 \leq \sqrt{\frac{d_\tau}{d}}$ by the existence of $\th_c$.
                           
                           When $R$ is decided by $L$,
                          one only needs to  note that $\alpha$ is also decided by $L$
                          and  the statement follows according to the proof of Theorem \ref{t2}.
                           \end{proof}

                           Instead of successively multiplying spheres, one can also consider using other kinds of links of class (I) (or their mixed types). 
                           For example, let $\tilde M$ be a focal submanifolds with ${\tt g}=4$ of sufficiently large dimension $n$.
                           Since  normal radius of each of them is always $\frac{\pi}{2}$,
                           we get the following.
                           
                                                      \begin{thm}\label{t4}
                           Let $L^k$ be the link of a Type-c cone  of class (II)  and $\frac{n}{2}\geq k\geq 3$.
                           Then $L\times \tilde M^n$ can span a Type-c cone.
                           \end{thm}
   \begin{proof}
 By the same kind of reasoning  in the proof of Theorem \ref{t3}.
   \end{proof}
 \begin{rem}
                           One can take $L$ to be a resulting Type (II) cone from Theorem \ref{main}.
                           Then Theorems \ref{t3}, \ref{manyS} and \ref{t4}
                           explore some extra structures.
                           Moreover, if $L$ is of class (b),
                           then so are the resultings in these theorems. 
                            \end{rem}

                           \begin{rem}\label{Frk}
                           We would like to point out 
                           that 
                           Lemma \ref{ThcE},
                           Theorem \ref{t2}, Remarks \ref{classb}, Corollary \ref{t2c}, Theorems \ref{t3}, \ref{manyS} and  \ref{t4}
                           will  also be true if ``Type-c" is replaced by ``Type-F".
                           The key point is to verify \eqref{cpF3} and \eqref{tobealpha} with the extra factor from \eqref{ratiotF}.
                           We provide details in appendix.
                           \end{rem}
                           
                           \begin{rem}
                           When ``Type-c" is replaced by ``Type-F",
                           the requirement the total dimension $\geq 11$ in Lemma \ref{ThcE} 
                           and
                           Theorem \ref{manyS}
                           can be
                           replaced 
                           by  the total dimension $\geq 7$,
                           since 
                           corresponding upper bound of replaced \eqref{cmF0} exists as $\th_F(8,\sqrt 7)<16^\circ$ according to Lawlor's table in \cite{Law}.
                           Also see  $(\ast)$ in appendix for remaining the requirement ``$k_1, k_2\geq 3$" unchanged in the Type-F version of Theorem \ref{t2}.
                           \end{rem}

                Based on 
                   the configuration results of this paper,
                   in  \cite{z0} 
                   we
                  get a constellation of
                   area-minimizing cones of codimension larger than one
                   which
                   cannot be calibrated by any globally defined continuous calibrations.
                        
                           {\ }    

{\ }


\section{Appendix}

We collect  proofs of Lemma \ref{tF}, Footnote \ref{ft} and Remark \ref{Frk} in appendix.

                                 \subsection{Proof of Lemma \ref{tF}. }
                                 
                                 As briefly mentioned in \S 2 that Lawlor's construction of area-decreasing projection is somehow equivalent to
                                 an inequality  for the dual form to be a calibration.
                                 With $t=\tan\th$ and $r$ the dimension of the cone,
                                 the inequality for calibration form is
                                 the following
                                 $$
                                 \left(g-\frac{tg'}{r}\right)^2
                                        +
                                              \left(\frac{g'}{r}\right)^2\leq 
                                                 \det
\left(
                                                                        \textbf{I}-t \, \textbf{h}_{ij}^v
                                                                        \right)^2,
                                                                        \text{ \ \ \ \ \ \ with } g(0)=1\ \ 
                                                                       \text{ and \ \ } g 
                                                                       \text{\  reaches }0.
                                 $$
                                 As Lawlor uses a uniform curve $\gamma$, 
                                 the above inequality becomes
                                  $$
                                 \left(g-\frac{tg'}{r}\right)^2
                                        +
                                              \left(\frac{g'}{r}\right)^2\leq 
                                                \inf_{x\in L, \, v\in S_x}
  \det
\left(
                                                                        \textbf{I}-t \, \textbf{h}_{ij}^v
                                                                        \right)^2
                                 $$
                                 and being a Type-F cone means that there exists $g(t)$ such that
        \begin{equation} \label{Lawlorineq}
                                 \left(g-\frac{tg'}{r}\right)^2
                                        +
                                              \left(\frac{g'}{r}\right)^2\leq 
    \left(F(\alpha, t, r)\right)^2,
     \text{ \ \ \ \ \ \ with } g(0)=1\ \
                                                                       \text{ and \ \ } g 
                                                                       \text{\  reaches }0.
                               \end{equation}
                                 Recalling \eqref{F}, we want to compare
                                   \begin{equation} 
                                                      F(\alpha, t, r)=
                                                                 \left(
                                                                 1-\alpha t\sqrt{\frac{r-1}{r}}
                                                                  \right)
                                                                  \left(
                                                                  1+\frac{\alpha t}{\sqrt{r(r-1)}}
                                                                  \right)^{r-1},
                                        \end{equation}
                                        and
                                         \begin{equation} 
                                                      F(\beta, s, m)=
                                                                 \left(
                                                                 1-\beta s\sqrt{\frac{m-1}{m}}
                                                                  \right)
                                                                  \left(
                                                                  1+\frac{\beta s}{\sqrt{m(m-1)}}
                                                                  \right)^{m-1}.
                                        \end{equation}
                                        Set $x=\alpha t\sqrt{\frac{r-1}{r}}=\beta s\sqrt{\frac{m-1}{m}}$
                                        and they have simpler expressions
                                         \begin{equation} \label{two}
                                                                 \left(
                                                                 1-x
                                                                  \right)
                                                                  \left(
                                                                  1+\frac{x}{{r-1}}
                                                                  \right)^{r-1}
                                                                  \text{\ \ \  and \ \ \ }
                                                                   \left(
                                                                 1-x
                                                                  \right)
                                                                  \left(
                                                                  1+\frac{x}{{m-1}}
                                                                  \right)^{m-1}.
                                        \end{equation}
                                        Note that, with fixed $x>0$,
                                   \begin{equation} \label{dn}
                                        \frac{d}{dn} \left(\log \left(1+\frac{x}{n}\right)^n\right)
                                        =\log\left(1+\frac{x}{n}\right)+n\cdot \frac{-\frac{x}{n^2}}{1+\frac{x}{n}}
                                        =\log \left(1+\frac{x}{n}\right)-\frac{x}{n+x}.
                           \end{equation} 
                                        Now taking derivative with respect to $x$ we have
                           \begin{equation*} 
                                        \frac{d}{dx}                                         
                                        \left\{\log \left(1+\frac{x}{n}\right)-\frac{x}{n+x}\right\}
                                        =
                                        \frac{\frac{1}{n}}{1+\frac{x}{n}}+\frac{-n-x+x}{(n+x)^2}
                                        =\frac{1}{n+x}-\frac{n}{(n+x)^2}
                                        =\frac{x}{n+x}>0.
                               \end{equation*} 
                                        As \eqref{dn} is zero at $x=0$,
                                        we know \eqref{dn} is positive when $x>0$.
                                        Hence the latter in \eqref{two} is larger than the former since $m>r$.
                                        
                                        Inspired by the proof of Lemma \ref{tc} in \cite{Law},
                                        we choose $\beta$ such that
                                        $$
                                    t=\frac{\beta}{     \alpha} s\sqrt{\frac{(m-1)r}{m(r-1)}}=\frac{m}{r}s,
                                        $$
                                       namely,
                                           $$
                                 \beta 
                                 \sqrt{\frac{(m-1)r}{m(r-1)}}=\frac{m}{r}{     \alpha}.
                                        $$
                                        
                                        Set $h(s)=g(t)=g(\frac{m}{r}s)$.
                                        Then $\dot h=\frac{m}{r}g'(\frac{m}{r}s)$.
                                        It follows from \eqref{Lawlorineq} for $g, r$ and $\alpha$ that
                                        $$
                                        \left(
                                        h(s)-\frac{s\dot h}{r}
                                        \right)^2
                                              +
                                                   \left(
                                                   \frac{\dot h}{m}
                                                   \right)^2
                                                   \leq 
                                                   \left(F(\alpha, t, r)\right)^2
                                                   \leq \left(F(\beta, s, m)\right)^2.
                                        $$
                                   Since $\dot h\leq 0\leq h$, we have
                                       $$
                                       \left(
                                        h(s)-\frac{s\dot h}{m}
                                        \right)^2
                                              +
                                                   \left(
                                                   \frac{\dot h}{m}
                                                   \right)^2
                                                   \leq \left(F(\beta, s, m)\right)^2.
                                        $$
        This means that $h$ solves the calibration inequality and hence, by $t=\frac{m}{r}s$, we get
                                     \begin{equation}
                                                      \tan \left(
                                                                      \th_F
                                                                      \left(m, \frac{m}{r}\sqrt{\frac{(r-1)m}{(m-1)r}}  \alpha   \right)
                                                                      \right)
                                                      <
                                                       \frac{r}{m}
                                                      \tan 
                                                              \left(
                                                              \th_F
                                                                \left(
                                                                r, \alpha
                                                                \right)
                                                              \right)  .
                         \end{equation}
                         To compare with \eqref{ratiotc} one can move the extra factor (inverse) to the right hand side and end up with Lemma \ref{tF}.
   
                                 {\ }
                           
                           \subsection{Proof of Footnote \ref{ft}.}
                           To show Footnote \ref{ft} about the monotonicity property for $\th_c(r,\cdot)$ in the second slot,
                           one just needs to observe that $(1-\alpha t)e^{\alpha t}$ in \eqref{control} is decreasing in $\alpha$.
                           So, suppose that
that
    $$
                                       \left(
                                        h(s)-\frac{s\dot h}{r}
                                        \right)^2
                                              +
                                                   \left(
                                                   \frac{\dot h}{r}
                                                   \right)^2
                                                   \leq \left((1-\beta s)e^{\beta s}\right)^2
$$                          
                           holds.
                           By setting $t=s$ and $g(t)=h(s)$, when $\beta>\alpha$ one has
                             $$ \left(
                                        g(t)-\frac{t\dot g}{r}
                                        \right)^2
                                              +
                                                   \left(
                                                   \frac{g'}{r}
                                                   \right)^2
                                                    \leq \left((1-\beta s)e^{\beta s}\right)^2
                             < \left((1-\alpha t)e^{\alpha t}\right)^2.
                             $$
                             Therefore, $\th_F(r, \alpha)<\th_F(r, \beta)$.
                            
                            \begin{rem} 
                            The increasing monotonicity for $\th_F(r,\cdot)$ follows similary.
                            \end{rem}
                            
                           {\ }

                                 \subsection{Proof of Remark \ref{Frk}. }
                                 Note that  to have \eqref{tobealpha} for Type-F
                               we need to check 
                                                                $$
                                 \sqrt
                                 {
                                 \frac{(k_1+k_2+1)k_1}{(k_1+k_2)(k_1+1)}
                                 }
                                 \frac{\sqrt{k_1+k_2}+1\big/\sqrt{k_1+k_2}}{\sqrt {k_1}+1\big/\sqrt {k_1}}>1
                                 $$
                                 which is equivalent to
                                       $$
                                 \sqrt
                                 {
                                 \frac{k_1+k_2+1}{k_1+1}
                                 }
                                 \frac{1+\frac{1}{k_1+k_2}}
                                {1+\frac{1}{k_1}}>1.
                                 $$
                                 For this we check whether $\frac{(x+1)^3}{x^2}$ is increasing in $x$.
                                 By differentiating $3\log (x+1)-2\log x$,
                                 we have $\frac{3}{x+1}-\frac{2}{x}=\frac{x-2}{x(x+1)}> 0$ when $x>2$.
                                 Hence the remark gets proved
                                 $(\ast)$ when $k_1\geq 2$ or when $k_1=1$ and $k_2\geq 4$.
                                  
                                 Due to exactly the same reason \eqref{cpF3} and with extra factor added
                                 also holds.
                                  
                                  {\ }
                                               

                              {\ }

{\ }


\begin{bibdiv}
\begin{biblist}




\bib{A}{article}{
    author={{Almgren, Jr.}, Frederick  J.}
    title={Some interior regularity theorems for minimal surfaces and an extension of Bernstein's theorem},
    journal={Ann. Math.},
    volume={84},
    date={1966},
    pages={277--292},
}




\bib{B}{article}{
    author={Bernstein}
    title={Sur un théoréme de géometrie et ses applications aux équations aux dérivées partielle du type elliptique},
    journal={Comm. de la Soc. Math de Kharkov},
    volume={15},
    date={1915},
    pages={38--45},
}



\bib{BdGG}{article}{
    author={Bombieri, Enrico},
     author={De Giorgi, Ennio},
      author={Giusti, E.},
    title={Minimal cones and the Bernstein problem},
    journal={Invent. Math.},
    volume={7},
    date={1969},
    pages={243--268},
}





\bib{CM}{article}{
    author={Carberry, Emma},
    author={McIntosh, Ian},
    title={Minimal Lagrangian 2-tori in $\mathbb CP^2$ come in real families of every dimension},
    journal={J. Lond. Math. Soc.},
    volume={69},
    date={2004},
    pages={531--544},
}





\bib{Ch}{article}{
    author={Cheng, B.N.},
    title={Area-minimizing cone-type surfaces and coflat calibrations},
    journal={Indiana Univ. Math. J.},
    volume={37},
    date={1988},
    pages={505--535},
}

 
 \bib{CH}{article}{
    author={Choe, Jaigyoung},
    author={Hoppe, Jens},
    title={Some minimal submanifolds generalizing the Clifford torus},
    journal={Math. Nach.},
    volume={291},
    date={2018},
    pages={2536--2542},
}



\bib{de2}{article}{
    author={De Giorgi, Ennio},
    title={Una estensione del teorema di Bernstein},
    journal={Ann. Sc. Norm. Sup. Pisa},
    volume={19},
    date={1965},
    pages={79--85},
}


 

\bib{FF}{article}{
    author={Federer, Herbert},
    author={Fleming, Wendell H.},
    title={Normal and integral currents},
    journal={Ann. Math. },
    volume={72},
    date={1960},
    pages={458--520},
}



\bib{F}{book}{
    author={Federer, Herbert},
    title={Geometric Measure Theory},
    place={Springer-Verlag, New York},
    date={1969},
}



    \bib{FK}{article}{
        author={Ferus, Dirk},
author={Karcher, Hermann},
   title= {Non-rotational minimal spheres and minimizing cones},
journal={Comment. Math. Helv.},
    volume={\bf 60},
    date={1985},
   pages={247--269},
   }
   

   
    \bib{fle}{article}{
author={Fleming, Wendell H.},
   title= {On the oriented Plateau problem},
journal={Rend. Circolo Mat. Palermo},
    volume={\bf 9},
    date={1962},
   pages={69--89},
   }






\bib{HS}{article}{
    author={Hardt, Robert},
    author={Simon, Leon},
    title={Area minimizing hypersurfaces with isolated singularities},
    journal={J. Reine. Angew. Math.},
    volume={362},
    date={1985},
    pages={102--129},
}

\bib{HL}{article}{
    author={Harvey, F. Reese},
    author={{Lawson, Jr.}, H. Blaine},
    title={Calibrated geometries},
    journal={Acta Math.},
    volume={148},
    date={1982},
    pages={47--157},
}









\bib{JCX}{article}{
    author={Jiao, Xiaoxiang},
    author={Cui, Hongbin},
    author={Xin, Jialin},
    title={Area-minimizing cones over products of Grassmannian manifolds},
    journal={Calc. Var. PDE},
    volume={61},
    date={2022},
    pages={205},
}



\bib{Law0}{article}{
    author={Lawlor, Gary R.},
    title={The angle criterion},
   journal={Invent. Math.},
   volume={95},
   date={1989},
       pages={437--446},
}


\bib{Law}{book}{
    author={Lawlor, Gary R.},
    title={A Sufficient Criterion for a Cone to be Area-Minimizing},
   place={Mem. of the Amer. Math. Soc.},
   volume={91},
   date={1991},
}






\bib{BL}{article}{
    author={{Lawson, Jr.}, H. Blaine},
    title={The equivariant Plateau problem and interior regularity},
    journal={Trans. Amer. Math. Soc.},
    volume={173},
    date={1972},
    pages={231-249},
}


\bib{L-Z}{article}{
    author={Li, Haizhong},
     author={Zhang, Yongsheng},
    title={Spiral Minimal Products}
    journal={arXiv: 2306.03328},
}










\bib{OS}{article}{
    author={Ohno, Shinji},
    author={Sakai, Takashi},
    title={Area-minimizing cones over minimal embeddings of R-spaces},
    journal={Josai Math. Monogr.},
    volume={13},
    date={2021},
    pages={69--91},
}






\bib{JS}{article}{
    author={Simons, James},
    title={Minimal varieties in riemannian manifolds},
    journal={Ann. of Math.},
    volume={88},
    date={1968},
    pages={62--105},
}



\bib{NS}{article}{
    author={Smale, Nathan},
    title={Singular homologically area minimizing surfaces of codimension one in Riemannian manifolds},
    journal={Invent. Math.},
    volume={135},
    date={1999},
    pages={145-183},
}





\bib{TZ}{article}{
    author={Tang, Zizhou},
    author={Zhang, Yongsheng},
    title={Minimizing cones associated with isoparametric foliations},
    journal={J. Diff. Geom.},
    volume={115},
    date={2020},
    pages={367--393},
} 






\bib{X}{book}{
    author={Xin, Yuanlong},
    title={Minimal submanifolds and related topics},
    place={Nankai Tracts in Mathematics, World Scientific Publishing},
   date={2003 (and Second Edition in 2018)},

}


\bib{XYZ2}{article}{
author={Xu, Xiaowei}
author={Yang, Ling}
   author={Zhang, Yongsheng},
   title={New area-minimizing Lawson-Osserman cones},
    journal={Adv. Math.},
   Volume={330},
   date={2018},
    pages={739--762},
   }







%


\bib{Z12}{article}{
   author={Zhang, Yongsheng},
   title={On extending calibration pairs}
   journal={Adv. Math.},
    volume={308},
    date={2017},
    pages={645--670},
   }

\bib{z}{article}{
        author={Zhang, Yongsheng},
    title={On realization of tangent cones of homologically area-minimizing compact singular submanifolds},
    journal={J. Diff. Geom.},
    volume={109},
    date={2018},
    pages={177--188},
    }
    
    
\bib{z0}{article}{
        author={Zhang, Yongsheng},
    title={Detect duality obstruction of calibrations in smooth category},
    	journal={arXiv:2512.04789}
    }




    
\end{biblist}
\end{bibdiv}
\end{document}